\documentclass{amsart}

\usepackage{amsthm}
\usepackage{amsfonts}
\usepackage{amsmath}
\usepackage{amscd}
\usepackage{graphicx}
\usepackage{enumerate}
\usepackage{amssymb}	
\usepackage{verbatim}
\usepackage{cite}
\usepackage{tikz-cd}

\usetikzlibrary{matrix,arrows,decorations.pathmorphing}

\DeclareMathAlphabet{\mathpzc}{OT1}{pzc}{m}{it}

\newtheorem{theorem}{Theorem}[section]
\newtheorem{proposition}[theorem]{Proposition}
\newtheorem{lemma}[theorem]{Lemma}
\newtheorem{corollary}[theorem]{Corollary}

\theoremstyle{definition}

\newtheorem{definition}{Definition}[section]

\theoremstyle{remark}
\newtheorem{remark}{Remark}

\newcommand{\R}{\mathbb{R}}
\newcommand{\Z}{\mathbb{Z}}
\newcommand{\C}{\mathbb{C}}
\newcommand{\Q}{\mathbb{Q}}
\newcommand{\N}{\mathbb{N}}

\newcommand{\tth}{^{\mbox{\rm{\scriptsize{th}}}}}

\newcommand{\abs}[1]{\left| #1 \right|}
\newcommand{\of}{\circ}

\newcommand{\ve}{\varepsilon}

\newcommand{\norm}[1]{\abs{\abs{#1}}}
\newcommand{\set}[1]{\left\lbrace #1 \right\rbrace}

\newcommand{\mc}{\mathcal}
\newcommand{\mf}{\mathfrak}

\def\id{\operatorname{id}}
\def\Diff{\operatorname{Diff}}
\def\Aut{\operatorname{Aut}}

\def\Lie{\operatorname{Lie}}

\def\St{\operatorname{St}}
\def\Hom{\operatorname{Hom}}
\def\Stab{\operatorname{Stab}}

\newcommand{\gen}{($\mf G$)}

\title[On the Rigidity of Weyl Chamber Flows]{On the Rigidity of Weyl Chamber Flows and Schur Multipliers as Topological Groups}
\author{Kurt Vinhage}

\begin{document}

\begin{abstract}
We effectively conclude the local rigidity program for generic restrictions of partially hyperbolic Weyl chamber flows. Our methods replace and extend previous ones by circumventing computations made in Schur multipliers. Instead, we construct a natural topology on $H_2(G,\Z)$, and rely on classical Lie structure theory for central extensions.
\end{abstract}

\maketitle

\section{Introduction}

The results of this paper make significant progress for the program for local rigidity of partially hyperbolic algebraic actions of $\R^k \times \Z^l$ with $k+l \ge 2$. In particular, it settles the rigidity problem for a broad class of actions on semisimple groups. An action $\alpha$ of a group $A$ on a homogeneous space $M = G / H$ is {\it homogeneous} if the action is given by some homomorphism $\varphi : A \to G$ and:

\[\alpha^a(gH) = \varphi(a)gH \]

Faithful homogeneous actions can be identified with subgroups of the Lie group $G$ isomorphic to $A$.

Such actions are among the principal sources of dynamical systems. Geodesic flows on surfaces of constant curvature, linear flow on tori, circle rotations, suspensions of toral automorphisms, and unipotent flows all fit into this class. A $C^1$ action of $A$ is {\it (pointwise) partially hyperbolic} if there is a splitting of the tangent bundle invariant under the action, and a distinguished central bundle, and a dense set of the elements of the action act hyperbolically with respect to the central bundle at each point. If the central bundle is integrable, the action is called dynamically coherent, and a dense set of the acting group acts normally hyperbolically with repsect to the corresponding foliation.

In the homogeneous case, the splitting of the tangent bundle corresponds exactly to the eigenspaces of the adjoint action of $A$ on $\Lie(G)$. The Lyapunov exponents correspond to the logarithms of the absolute values of the corresponding eigenvalues.

An algebraic action $\alpha : A \to \Diff^r(M)$ is $C^{l,m,n}$-locally rigid if any $C^l$ action $\alpha'$ close in the  $C^m$ topology is conjugate to a nearby algebraic model via a $C^n$ diffeomorphism. Note that a reparameterization, or perturbation of the embedding $A \to G$ can change the dynamical invariants (like Lyapunov exponents) under $C^1$ conjugates, so we must include such possibilities.

There are two principal sources of rigid partially hyperbolic action. One is the suspesnsion of an action by toral automorphisms, which can be obtained as a homogeneous flow on a solvmanifold. Such systems are better studied as discrete-time systems. The other principal source is semisimple groups, in which one takes a subgroup of an $\R$-split Cartan subalgebra as the acting group. Such actions are called (partially hyperbolic) Weyl chamber flows, and are the subject of this paper.

Fully hyperbolic systems were demonstrated to be $C^{\infty,1,\infty}$ rigid in \cite{ks97} by Katok and Spatzier. In the partially hyperbolic setting, irreducible actions by commuting toral automorphisms were shown to be $C^{\infty,l,\infty}$ rigid for some $l$ in the work of Damjanovic and Katok \cite{kdtoral}. The semisimple case, however, proved more resistant.

Let $G$ be a semisimple Lie group with each of its simple factors of real rank $k_i \ge 2$ (such a Lie group is {\it genuinely higher-rank}), $\Gamma$ be an irreducible, cocompact lattice, and $\R^k \cong N_A \subset G$ be a $\R$-split Cartan subgroup, and $N_K$ be the compact part of the centralizer of $N_A$, so that $N = C_G(N_A) = N_A \times N_K$. The Weyl chamber flow associated to the triple $(G,\Gamma,H)$ is the left action of $N_A$ on the algebraic factor space $M = N_K \backslash G/\Gamma$. Such actions are fully hyperbolic, and proved rigid in \cite{ks97}.

Later, partially hyperbolic actions were studied by making two modifications: First, we consider {\it partially hyperbolic} Weyl Chamber flows (PHWCF) by instead using the phase space $M = G / \Gamma$ (ie, we do not kill the compact part of the centralizer). Second, we consider restrictions of the full action, under a genericity assumption. Let $A \cong Z^k \times \R^l$ be a subgroup of the split Cartan subgroup $N_A$, $\Delta$ be the set of roots of $\mf g = \Lie(G)$ with respect to $\mf h = \Lie(N_A)$, and $\mf a \subset \mf h$ be the smallest subalgebra such that $\exp(\mf a) \supset A$. A {\it generic restriction} of a PHWCF is a triple $(G,\Gamma,A)$, with $G$ and $\Gamma$ and the action as before, and:

\vspace{.25cm}

\begin{itemize}
\item[\gen]  If $r_1,r_2 \in \Delta$ satisfy $r_1|_{\mf a} = \lambda r_2|_{\mf a}$, then $r_1 = \lambda r_2$
\end{itemize}

\vspace{.25cm}

This is the {\it genuinely higher-rank} condition, although it is a slightly stronger restriction than that which appears elsewhere in the literature. It is a natural assumption to ensure that two roots do not collapse into one projectivized root. In such an event, the number of Weyl chambers in $A$ may increase in even standard homogeneous perturbations, changing the structure of the path and cycle spaces. In certain cases, the condition can be relaxed (see \cite{zwang-2}), but the general argument here requires persistence of the Lyapunov structures.

\begin{theorem}[The Main Theorem]
\label{main}
Let $\alpha_0 : A \to \Diff^{\infty}(G/\Gamma)$ be a generic restriction of a PHWCF. Then if $\alpha$ is a $C^1$-small, $C^\infty$ perturbation of $\alpha_0$, there exists a homomorphism $\iota : A \to N$, close to the inclusion, and $C^\infty$ diffeomorphism $h : M \to M$ such that:

\[ \alpha^a \of h(x) = h \of \alpha_0^{\iota(a)}(x)\]
\end{theorem}

One may try to draw comparisons between the higher rank case and rank 1 case (flows and iterated transformations), where fully hyperbolic actions are called Anosov. In this setting, the correct notion is structural stability, which gives orbit equivalence via a H\"older homeomorphism. Such equivalence is a conjugacy in the case of $\Z$ actions. In the case of flows, structural stability only gives H\"older rigidity up to time change.

The higher rank assumption therefore plays a crucial role. Indeed, if one looks at individual diffeomorphisms in the action, they may fail to be structurally stable. But perturbations with changed behavior will have small centralizer, in particular will not fit into an abelian group action of higher rank.

We summarize the history of the semisimple case briefly. Damjanovic and Katok were able to show $C^{\infty,2,\infty}$ local rigidity for such actions on $SL(d,\R)$ and $SL(d,\C)$. Their result has been extended to many other cases of simple Lie groups by Damjanovic and Wang:

\begin{enumerate}[(i)]
\item $G$ $\R$-split \cite{damjanovic07}
\item $G$ complex simple \cite{damjanovic07}
\item $G = SO(m,n),SU(m,n)$, $m,n \ge 2$ \cite{zwang-1}
\item $G = Sp(n,\R)$, Certain nongeneric restrictions \cite{zwang-2}
\end{enumerate}

Note that our result improves the previous ones in a few ways: First, we use persistence of local transitivity under $C^1$ perturbations to improve $C^{\infty,2,\infty}$ rigidity to $C^{\infty,1,\infty}$ rigidity. Second, we replace key steps of the existing argument verified ad-hoc for each group with a unified approach.

All of the previous proofs followed the same general scheme: First, apply normal hyperbolicity theory to obtain a H\"older leaf conjugacy for the central foliations. Then, prove cocycle rigidity over perturbations to correct the error of the leaf conjugacy. Finally, a standard argument for algebraic partially hyperbolic actions shows that any $C^0$ conjugacy must in fact be $C^\infty$.

The main obstacle to overcome has been cocycle rigidity. Cocycle rigidity for each action was solved using algebraic K-theory and Schur multipliers. Until now, methods were restricted to showing certain properties of these invariants for each simple Lie group separately. This paper provides a unifying argument, which both replaces and extends the previous work relying on these methods.


\subsection{Structure of the Paper}

In sections \ref{lyap-sec} and \ref{pcf-sec}, a series of standard reductions are made to reduce the problem of local rigidity to a problem of cocycle rigidity for the perturbed action $\alpha$ (see Proposition \ref{correction-cocycle}, Proposition \ref{pcf-suff}, Corollary \ref{lattice-finite}, and Proposition \ref{homeo}).  An astute reader will notice that such reductions only guarantee topological conjugacy. We then apply a method of nonstationary normal forms which guarantees that restricted to each leaf of the Lyapunov foliations, the map is $C^\infty$. Then, the fact that the distributions to these foliations generate the entire tangent space at each point, we get that the map is $C^\infty$ (see \cite{ks97} for further details of this method). 

Upon completing the reductions, the main difficulty is showing vanishing of the periodic cycle functional for contractible cycles in the case when $G$ is simple (we will be to handle the simple case first (Theorem \ref{simple-case}), and prove Theorem \ref{main} using semisimple structure). In section \ref{schur-sec}, properties of the contractible cycle group (the local domain of the periodic cycle functional) are established for the homogeneous action. 

In section \ref{holonomy-sec}, correspondence between contractible cycles for the homogeneous and perturbed actions are constructed. Such an association will guarantee that vanishing for homomorphisms from the homogeneous cycle group also imply it for the perturbed contractible cycle group. In section \ref{proofs-sec}, the proofs of the two rigidity theorems (Theorems \ref{main} and \ref{cocyle-rig}) are given. Section \ref{next-steps-sec} contains comments on remaining cases and the current state of the local rigidity program.

\vspace{.25cm}

\noindent {\it Acknowledgements.} The author would like to thank his PhD advisor Anatole Katok, Federico Rodriguez-Hertz, Zhenqi Wang and Nigel Higson for many illuminating conversations, guidance, and forthcoming assistance.

\section{Lyapunov Exponents, Hyperplanes and Foliations}
\label{lyap-sec}

If $\alpha : \R^k \times \Z^l \to \Diff^1(M)$ is a higher-rank action preserving an ergodic measure $\mu$, then there exists a splitting of the tangent bundle $TM = \displaystyle\bigoplus_{r \in \Delta} E^r$ such that if $v \in E^r$ is a unit vector:

\[ \lim_{n \to \pm\infty} \frac{1}{n}\log \norm{d\alpha^{na}(v)} = r(a) \]

$\Delta \subset \Hom(A,\R)$ are called the {\it Lyapunov exponents} for the action $\alpha$. The kernels of the exponents $H_r = \ker r$, are called the {\it Lyapunov hyperplanes}. The connected components of $\R^{k+l} \setminus \bigcup H_r$ are called the {\it Weyl chambers} the action. Elements contained in a Weyl chamber are called {\it regular}.

Any element $a$ has three distributions:

\begin{eqnarray*}
E_a^s & = & \bigoplus_{r : r(a)< 0} E^r \\
E_a^c & = & E^0 \\
E_a^u & = & \bigoplus_{r : r(a) > 0} E^r
\end{eqnarray*}

In our case, each distribution integrates to a foliation $W_a^s$ and by forming intersections, we can form the {\it coarse Lyapunov distributions} $W^r = \displaystyle\bigcap_{a : r(a) < 0} W_a^s$. This leads to the {\it cycle group} $\mc C_{x_0}(\mc W)$, defined as follows. Fix a base point $x_0$. A {\it Lyapunov path} based at $x_0$ is a sequence of points $(x_0,\dots,x_N)$ such that $x_{i+1} \in W^{r_i}(x_i)$ for some $r_i$. The path is a {\it cycle} if $x_N = x_0$. The space of cycles $\mc C_{x_0}(\mc W)$ forms a group under concatenation (where we allow cancellation of a cycle with its reverse). $\mc C_{x_0}(\mc W)$ carries a natural topology, by writing it as $\mc C_{x_0}(\mc W) = \bigsqcup_{n \in \N_0} \mc C^n_{x_0}(\mc W)$ where $\mc C^n_{x_0}(\mc W)$ is the set of cycles of length $n$. Each $\mc C^n_{x_0}(\mc W)$ carries a natural topology inherited from $M^n$. Furthermore, let $\mc C^c_{x_0}(\mc W)$ be the set of cycles which are contractible in $M$ and $\mc C^G_{x_0}(\mc W)$ denote the set of cycles which lift to cycles in the group $G$.

A cycle $(x_0,\dots,x_N = x_0)$ is {\it stable} if there exists $a \in A$ such that $r_i(a) < 0$ for every $i = 0,\dots,N-1$. Let $\mc S_{x_0}(\mc W)$ denote the normal closure of the group generated by stable cycles.

To carry these structures to the perturbed action, we use normal hyperbolicity theory. It comes in many flavors, but the results we need can be found in \cite{pugh-shub-wilk}, namely:

\begin{theorem}
\label{normal-stability}
Let $f : M \to M$ act normally hyperbolically with respect to a $C^1$ foliation $\mc N$. Then if $g$ is a sufficiently $C^1$-small perturbation of $f$, then $g$ acts normally hyperbolically with respect to a unique foliation $\mc N_g$ close to $\mc N$, and $\mc N$ and $\mc N_g$ are leaf conjugate via a H\"older homeomorphism of $M$.
\end{theorem}

Because the acting group is abelian, this procedure done for a single regular element will work for the entire action. By choosing an element of each Weyl chamber, we can also integrate the stable and unstable manifolds for perturbations of each such element (for sufficiently small $C^1$ perturbations). By taking intersections, we can then form the Lyapunov foliations for such a perturbation and carry them to H\"older foliations via the leaf conjugacy described in Theorem \ref{normal-stability}. Thus, all of the smooth structures for the algebraic action can be obtained as H\"older structures for the perturbation $\alpha$.

\section{Cocycle Rigidity and the Periodic Cycle Functional}
\label{pcf-sec}

A series of reductions outlined in \cite{kd2011} reduce the proof of Theorem \ref{main} to a problem of cocycle rigidity for the original action $\alpha_0$. The technique is outlined as follows.

First, a conjugacy up to the neutral foliation is constructed, provided by Theorem \ref{normal-stability}. Of course, the conjugacy is only guaranteed to be H\"older. Recall that a {\it cocycle taking values in $N$} over an action $\alpha : A \to \Diff(M)$ is a function $\beta : A \times M \to N$ satisfying:

\[ \beta(ab,x) = \beta(a,\alpha^b(x))\beta(b,x) \]

The following Proposition will help us reduce the problem of local rigidity to cocycle rigidity:

\begin{proposition}
\label{correction-cocycle}
If $\alpha_0$ is a PHWCF and $\bar{\alpha}$ is a $C^1$-small, $C^\infty$ perturbation, then there exists a H\"older homeomorphism $h_0 : M \to M$ such that $h_0\left(\mc N\right) = \mc N_0$, where $\mc N_0$ is the folitation of $M$ by the orbits of the Weyl chamber flow, and $\mc N$ is the neutral foliation for $\bar \alpha$. Furthermore, $\beta : A \times M \to N$ defined by:

\[ \beta(a,x) = \alpha^a(x)\alpha_0^a(x)^{-1} \]

is a H\"older cocycle over the action $\bar{\alpha}$ (or equivalently, $\alpha$), where $\alpha = h_0 \of \bar{\alpha} \of h_0^{-1}$.
\end{proposition}

\begin{proof}
The existence of the claimed H\"older homeomorphism, is just a particular case of Theorem \ref{normal-stability} applied to abelian actions. To see that $\beta$ is in fact a cocycle, observe that:

\[ \begin{array}{rcl}
\beta(a,\alpha^b(x))\beta(b,x) & = & \alpha^a(\alpha^b(x)) \alpha_0^a(\alpha^b(x))^{-1} \alpha^b(x) \alpha_0^b(x)^{-1} \\
 & = & \alpha^{ab}(x)(a\alpha^b(x))^{-1} \alpha^b(x)(bx)^{-1} \\
 & = & \alpha^{ab}(x) \alpha^b(x)^{-1}a^{-1}\left[\alpha^b(x)x^{-1}\right]b^{-1} \\
 & = & \alpha^{ab}(x)x^{-1}a^{-1}b^{-1} \\
 & = & \beta(ab,x)
\end{array} \]

We make heavy use of the fact that the orbit of $\alpha$ through $x$ is contained in $Nx$, and $A$ is central in $N$.

\end{proof}

The cocycle $\beta$ can be thought of as the ``defect'' of the homeomorphism $h_0$ from a conjugacy. Our goal is to correct this defect by showing that the cocycle is cohomologically trivial. Recall that a cocycle is {\it cohomologous to a constant} if there is a transfer map $H : M \to N$ and a homomorphism $s : A \to N$ such that:

\[ \beta(a,x) = H(\alpha^a(x))s(a)H(x)^{-1} \]

If we were to have this condition, we could define $h(x) = H(x)^{-1}x$ (which is a priori only continuous, not a homeomorphism). Then if $\iota : A \to N$ is defined by $\iota(a) = a\cdot s(a)$:

\[ 
\begin{array}{rcl}
\alpha_0^{\iota(a)} \of h(x) & = & \iota(a) H(x)^{-1} x \\
& = & s(a) H(x)^{-1} \alpha_0^{a}(x) \\
& = & s(a) H(x)^{-1}\beta(a,x)^{-1}\alpha^{a}(x) \\
 & = & H(\alpha^a(x))^{-1}\alpha^{a}(x) \\
 & = & h \of \alpha^{a}(x)
\end{array}
\]

\begin{proposition}[Proposisition 2.1, \cite{kd2011}]
\label{homeo}
The map $h$ is a homeomorphism
\end{proposition}

The proof of Proposition \ref{homeo} requires another statement proved in \cite[Theorem 3]{kd2011} (see Propsosition \ref{simple-transitive}) which deals with the holonomy group induced by the Lyapunov foliations. A proof of this Proposition immediately follows the proof of Proposition \ref{simple-transitive}.


\begin{theorem}
\label{cocyle-rig}
Let $\alpha : A \to \Diff^{\infty}(M)$ be a sufficiently $C^1$-small perturbation of a generic restriction of a PHWCF. Then if $\beta : A \times M \to N$ is a sufficiently small, H\"older continuous cocycle over $\alpha$ with $N$ either abelian or compact, then $\beta$ is cohomologous to a constant via a continuous transfer map $H$
\end{theorem}

The main tool in proving Theorem \ref{cocyle-rig} is the {\it periodic cycle functional} for higher-rank actions. For convenience, we let $\mc W$ denote the family of Lyapunov foliations for a general action $\alpha$.

\begin{definition}
Let $\alpha : A \to \Diff^1(M)$ be an action of $A \cong \R^k \times \Z^l$, and $\beta : A \times M \to H$ be a H\"older cocycle over $\alpha$. The periodic cycle functional is a continuous homomorphism $P_\beta : \mc C_{x_0}(\mc W) \to H$ which sends

\[ P_\beta(x_0,x_1,\dots,x_N = x_0) = \prod_{i= 0}^{N-1} p_\beta(x_i,x_{i+1}) \]

where $p_\beta(x,y) = \lim_{n \to \infty} \beta(na,x)^{-1}\beta(na,y)$, where $x \in \mc F_r(y)$ and $r(a) < 0$. $P_\beta$ is the restriction of a continuous map $F_\beta$ from the space of paths, defined the same way (although, in general, the space of paths with fixed base point has no group structure, so $F_\beta$ cannot be considered a ``homomorphism'').
\end{definition}

In \cite{kd2005}, it is shown that this definition is independent of choice of $a$ assuming that $r(a) < 0$. Note also that the cycle group may change, as $\mc C_{x_0}(\mc W)$ changes with the basepoint.

\begin{proposition}[Proposition 4(1), Proposition 10(1),\cite{kd2005}]
\label{pcf-suff}
Let $\alpha : A \to \Diff^{\infty}(M)$ be an action such that the Lyapunov foliations are transitive. Then a H\"older cocycle $\beta$ with values in a Lie group with a bi-invariant metric is cohomologous to a constant via a continuous transfer map $H$ if and only if $P_\beta : \mc C_{x_0}(\mc W) \to N$ is trivial for some $x_0 \in M$
\end{proposition}

\begin{remark}
In \cite{kd2005}, the Proposition as stated requires {\it local} transitivity of the Lyapunov foliations. In \cite{damjanovic07}, Remark 3.4, it is observed that this condition may be relaxed. This observation is particularly easy in our case: By Lemma \ref{section-exist}, to each point nearby $x$, we can associate continuously a Lyapunov path for the action $\alpha_0$, and through the holonomy projections, we can associate  Lyapunov paths for the action $\bar{\alpha}$ (which may or may not be very long!). Regardless of the length when realized as paths in the manifold, the paths themselves will have fixed combinatorics, and will vary continuously with the endpoint $x$.

We now recall that the transfer map is simply defined to be:

\[ H(x) = F_\beta(\rho_x) \]

where $\rho_x = (x_0,\dots,x_N)$ is a path from a fixed base point $x_0$ to $x_N = x$. This observation makes it obvious that transitivity is sufficient. 
\end{remark}

Cocycle rigidity over the unpertrubed flow $\alpha_0$ was shown in in the case of $SL(d,K)$ in \cite{kd2005}, using the periodic cycle functional and classical algebraic $K$-theory. By constructing projections between the cycle groups of the algebraic and perturbed actions, Katok and Damjanovic extended this to the perturbed action in \cite{kd2011}. We prove an analogous result in Propsoition \ref{holonomy-iso}. The action of holonomies is also the crucial tool in proving Proposition \ref{homeo}. In summary, proving Theorem \ref{cocyle-rig} will imply Theorem \ref{main}.

If $M$ is a compact manifold, and $\mc W$ is a family of foliations, let $\pi_x^{\mc W} : \mc P_x(\mc W) \to M$ denote the projection which sends a path to its endpoint (note that for a general family, this may not be surjective). The following will guarantee we still have transitivity of the Lyapunov foliations, and is a straightforward adaptation of the arguments for \cite[Proposition 1.4]{gps94} or \cite[Theorem 3.4]{pugh-shub97}:

\begin{lemma}
Suppose that $\mc F = \set{\mc F_i}$ is a family of foliations on a compact manifold $M$, such that for every $x \in M$, there is a continuous local section $\sigma_x : U_x \to \mc P_x(\mc F)$ for $\pi_x^{\mc F}$. Then if $\mc F'$ is a family of foliations sufficiently close to $\mc F$, $\mc F'$ also has a family of continuous section $\sigma_x' : U_x' \to \mc P_x(\mc F')$.
\end{lemma}

By compactness, we can cover the manifold $M$ with finitely many $U_x$ (and hence $U_x'$). Thus, if we fix one such basepoint $x_0$, we can reach any point of the manifold with a path of uniformly bounded length. In our case, we get the maps $\sigma_x$ from Lemma \ref{section-exist}, so:

\begin{proposition}
\label{local-transitive}
The Lyapunov foliations of $\bar{\alpha}$ are locally transitive
\end{proposition}

One of the main ingredients in proving the vanishing for the algebraic action is the observation:

\begin{lemma}
\label{stable-vanishing}
$\mc S_{x_0}(\mc W) \subset \ker P_\beta$ for every H\"older cocycle $\beta$
\end{lemma}

\begin{proof}
This is clear from the definition, since if $\sigma = (x_1,x_2,\dots,x_n)$ is a stable cycle, there exists a uniform $a \in A$ so that when $x_{i+1} \in W^r(x_i)$, $r(a) < 0$. But then:

\[ 
\begin{array}{rcl}
P_\beta(\sigma) & = & \displaystyle\prod_{i=0}^{n-1} p_\beta(x_i,x_{i+1}) \\
 & = & \displaystyle \lim_{n\to \infty} \displaystyle \prod_{i=0}^{n-1} \beta(ka,x_i)^{-1}\beta(ka,x_{i+1}) \\
 & = & \displaystyle \lim_{k \to \infty} \beta(ka,x_0)^{-1}\beta(ka,x_n) \\
 & = & e
 \end{array} \]
\end{proof}

Lemma \ref{stable-vanishing} allows us to consider $P_\beta$ as a homomorphism from $C_{x_0}(\mc W) / S_{x_0}(\mc W)$. We will show that it is trivial on $C^c_{x_0}(\mc W) / S_{x_0}(\mc W)$. It will prove vanishing of the periodic cycle functional, provided the following proposition and the Margulis normal subgroup theorem, since the centralizer of the Cartan always splits as a product of an abelian group and a compact group $N \cong N_A \times N_K$:

\begin{proposition}
\label{lattice-rig}
Fix a cocompact, irreducible lattice $\Gamma$ of a semisimple, genuinely higher-rank Lie group $G$, and a compact group Lie group $N_K$. Then there is a constant $C(\Gamma,N_K)$ such that if $\varphi : \Gamma \to N_K$ is a homomorphism, its image has order less than $C(\Gamma,N_K)$.
\end{proposition}

\begin{proof}
Any compact group $N_K$ can be represented as an algebraic matrix group. Let $\overline{\; \cdot \;}$ denote the standard topological closure and $\overline{\; \cdot\; }^z$ denote the Zariski closure. The Margulis Normal Subgroup theorem guarantees that $\overline{\varphi(\Gamma)}^z$ is a semisimple algebraic group defined over $\Q$. Assume that $\varphi(\Gamma)$ is not finite. It is a noncompact subgroup of $\overline{\varphi(\Gamma)}^z$, since if it were closed, it would be Lie and hence uncountable. Since the center  of $N_K$ is finite, we can assume that $\varphi$ extends to a homomorphism $\overline{\varphi} : \Gamma \to N_K / Z(N_K)$ (by quotienting in $\Gamma$ by $\varphi^{-1}(Z(N_K))$. Furthermore, the Margulis superrigidity theorem tells us that $\overline{\varphi}$ is the restriction of some homomorphism $\widetilde{\varphi} : G \to N_K / Z(N_K)$.

Since we assumed that the image of $\varphi$ is infinite, the kernel of $\widetilde{\varphi}$ cannot be all of $G$. Since $G$ is semisimple, $\widehat{G} = G / \ker \widetilde{\varphi}$ is again a semisimple Lie group. But this implies that $\Lie(\widehat{G}) = \Lie(N_K)$, which cannot be true, because $N_K$ is compact and $G$ has no compact factors.

Thus the image is always finite. To see that it is bounded by a fixed number, note that $N_K$ is a matrix group and so Jordan's theorem tells us that there is a fixed constant $C'$ such that any finite subgroup contains a normal abelian subgroup whose index is at most $C'$. Let $\Lambda$ be such a subgroup of $\varphi(\Gamma)$. Then $\Lambda_0 = \varphi^{-1}(\Lambda)$ is finite index in $\Gamma$, and the index is bounded by a fixed constant $C'$. But then $\Lambda_0$ is exactly the kernel of the homomorphism $\Gamma \to \Gamma / \Lambda_0$, so we only need to verify that there are finitely many such projections onto groups of order $\le C'$. This follows from finite generation of the lattice.

\end{proof}

\begin{corollary}
\label{lattice-finite}
Any homomorphism $\varphi :\Gamma \to N$ has finite image, whose order is bounded by a fixed constant
\end{corollary}

\begin{remark}
Such a Proposition was not needed in \cite{kd2011} or \cite{damjanovic07}, since in the split case $N_K = \set{e}$. Furthermore in \cite{damjanovic07}, the Corollary follows trivially, since $\pi_1(SL(d,\R))$ is finite, and any lattice in $SL(d,\R)$ has property (T).

Indeed, by Remark 6.20(1) in \cite{margulis91}, property (T) persists for lattices in the universal cover of algebraic groups, and hence generally for semisimple groups. See also Remark 3.8 in \cite{zwang-2}.
\end{remark}

\section{The Group of Cycles and Central Extensions}
\label{schur-sec}

In the case of Weyl chamber flows, the Lyapunov foliations and cycle group carry an important algebraic structure. Recall that a group is {\it perfect} if $G = [G,G]$. If $G$ is an abstract, perfect group, it has a {\it universal central extension} $1 \to K \to\widetilde G \to G \to 1$. It is characterized by the property that if $1 \to Z \to H \to G \to 1$ is any other central extension, then there exists a unique homomorphism $\Phi : \widetilde{G} \to H$ which extends the identity on $G$. The kernel of the projection coincides with the second group homology of $G$, $H_2(G,\Z)$ often called the {\it Schur multiplier}. It should be noted that for the needs of this paper, the second homology group for $G$ is taken as an abstract group and does not reference the topology of $G$.

Our aim in this section is to put appropriate algebraic and topological structures on the cycle spaces. Fix the basepoint $e$, so that $\mc P(\mc U) = \mc P_e(\mc U)$ is the set of Lyapunov paths for the action of $\alpha_0$ on $G$. Put a topology on $\mc P(\mc U)$ in a similar fashion to the cycle group: we write $\mc P(\mc U) = \sqcup_{n \in \N} \mc P_n(\mc U)$, where $\mc P_n(\mc U)$ are the paths of combinatorial length $n$.

\begin{proposition}
\label{genericity}
If $A \subset N_A \subset G$ satisfies the genericity condition \gen, then $\mc P(\mc U)$ carries a group structure. Furthermore, the structure makes it isomorphic to the free product of the groups $U_r = \exp\left(\bigoplus_{t\in \R_+} \mf g_{tr}\right)$, where $r \in \Delta_{N_A}$, where $\Delta_{N_A}$ are the roots of the split Cartan subalgebra $\log(N_A)$
\end{proposition}

\begin{proof}
We first show that the $U_r$ are exactly the coarse Lyapunov foliations for the action. They are the coarse foliations for the full action of $N_A$, and the assumption \gen \, guarantees that if $r_1|_{\log(A)} = \lambda r_2|_{\log(A)}$, then $r_1 = \lambda r_2$. Thus, they share the same projectivized roots and hence the same coarse Lyapunov foliations.

Unlike the cycle group, we cannot use concatenation to define multiplication, so we must take some care. If $\rho_1 = (e,g_1,\dots,g_n)$ is one path, and $\rho_2 = (e,h_1,\dots,h_k)$ is another, then define the product:

\[ \rho_1 * \rho_2 = (e,h_1,\dots,h_k,g_1h_k,g_2h_k,\dots,g_nh_k) \]

Thus an element of $\mc P(\mc U)$ corresponds to a sequence:

\[ \left(e, g_1^{(r_1)}, g_2^{(r_2)}g_1^{(r_1)}, \dots, g_n^{(r_n)}g_{n-1}^{(r_{n-1})}\dots g_1^{(r_1)} \right) \]

where $g_i^{(r_i)} \in U_{r_i}$. We allow cancellation of reverse paths, so an element of $\mc P(\mc U)$ can thus be identified with a formal product of elements of $\bigcup_{r \in \Delta} U_r$. That is, we have shown that $\mc P$ is the free product of the groups $\bigcup_{r \in \Delta} U_r$. 

\end{proof}

We make use of the fact that the Lyapunov foliations for the unperturbed action are extremely well-behaved. Note that this structure makes $\mc P(\mc U)$ a topological group with a canonical projection $\pi_e^{\mc U} : \mc P(\mc U) \to G$ which sends a path to its endpoint. The cycle group can then be characterized as the closed subgroup $\mc C^G(\mc U) = \ker \pi_e^{\mc U}$. Recall that $\mc C^G(\mc U)$ is the group of cycles which lift to cycles in $G$. If $G$ is simply connected, then these are exactly the contractible cycles.

\begin{proposition}
\label{cycles-are-schur}
The cycle group $\mc C^G(\mc U)/\mc S(\mc U)$ is a factor the Schur multiplier of $G$, considered without topology, and the group $\mc P(\mc U) / \mc S(\mc U)$ is a factor of the universal central extension of $G$.
\end{proposition}

\begin{proof}
By Proposition \ref{genericity}, the group of Lyapunov paths carries a canonical group structure, which coincides with the group structure on the cycles. This will fit into a central extension of $G$:

\[ 1 \to \mc C^G(\mc U) / \mc S(\mc U) \to \mc P(\mc U) / \mc S(\mc U) \to G \to 1 \]


We wish to factor $\mc P(\mc U)$ by $\mc S(\mc U)$. Instead, let us identify a subgroup of $\mc S(\mc U)$ which is easy to work with. Namely, consider the following elements of $\mc S(\mc U)$, which we assume are chosen so that they project to the identity:

\vspace{.25cm}

\begin{enumerate}[(i)]
\item ${\left[g_1^{(r_1)},g_2^{(r_2)}\right]}g^{(r_1+r_2)} \hspace{.5cm}  r_1,r_2,r_1+r_2 \in \Delta$ \vspace{.1cm}
\item ${\left[g_1^{(r)},g_2^{(r)}\right]}$ \vspace{.1cm}
\item ${\left[g_1^{(r_1)},g_2^{(r_2)}\right]} \hspace{1cm} r_1,r_2 \in \Delta, r_1 + r_2 \not\in \Delta$
\end{enumerate}

\vspace{.25cm}

These are formal products, which when multiplied in $G$ give $e$. We will call these elements the {\it commutator relations} on the Lyapunov generators, and denote the normal closure of the group generated by the commutator relations as $\mc S_{\mbox{\tiny comm}}$. Since $\mc S_{\mbox{\tiny comm}} \subset \mc S(\mc U)$, there is a natural projection $\mc P(\mc U) / \mc S_{\mbox{\tiny comm}} \to \mc P(\mc U) / \mc S(\mc U)$.

The group $\mc P(\mc U) / \mc S_{\mbox{\tiny comm}}$ appears in a classical paper of Deodhar, and is one of the fundamental Lemmas that give us an important structure theory:

\begin{lemma}[Theorem 1.9,\cite{deodhar78}]
\label{univ-centr}
If $G$ is a simple algebraic group and has $\R$-rank of $G$ at least 2, then $\mc P(\mc U) / \mc S_{\mbox{\tiny comm}}$ is the universal central extension of $G$, and $\mc C^G(\mc U) / \mc S_{\mbox{\tiny comm}}$ is the Schur multiplier of $G$
\end{lemma}

In fact, the assumption of $G$ being algebraic is unnecessary in Lemma \ref{univ-centr}. We note also if $G'$ is the universal cover of an algebraic group $G$, it differs from the algebraic version by exactly $\pi_1(G)$. That is, $G'$ is a perfect central extension of $G$, so there is a map $\widetilde{G} \to G'$ uniquely extending the identity on $G$. Since it is unique, it must be the map which associates to a path its fixed-endpoint homotopy class. Its kernel is exactly the contractible cycles. The universal property of $\widetilde{G}$ is easily checked for $G'$ as well (since any central extension of $G'$ is also a central extension of $G$), so $\widetilde{G}$ is also a universal central extension of $G'$. Hence, $H_2(G',\Z) \cong \mc C^c(\mc U) / \mc S_{\text{comm}}$.
\end{proof}

\begin{proposition}
\label{top-cent-ext}
If $G$ is a simple Lie group with a PHWCF, then the group $\mc P(\mc U) / \mc S(\mc U)$ is a perfect central extension of $G$, such that the following is a short exact sequence of topological groups:

\[ 1 \to \mc C^G(\mc U) / \mc S(\mc U) \to \mc P(\mc U) / \mc S(\mc U) \to G \to 1 \]

\end{proposition}

\begin{proof}
Note that because $\widetilde{G} = \mc P(\mc U) / \mc S_{\mbox{\tiny comm}}$ is a {\it universal} central extension of $G$, it must be perfect. We give a brief proof. The identity homomorphism must be the unique homomorphism of $\widetilde{G}$ which extends the identity on $G$, by the universal property of universal central extensions. Since $G$ is perfect, $[\widetilde{G},\widetilde{G}] \subset \widetilde{G}$ is also a central extension of $G$, so there exists a homomorphism $\widetilde{G} \to [\widetilde{G},\widetilde{G}]$ which extends the identity on $G$. Composing this with the inclusion of $[\widetilde{G},\widetilde{G}]$ into $\widetilde{G}$ will give a nonuniquness of the extension of the identity unless $\widetilde{G}$ is perfect. Since $\mc P(\mc U) / \mc S(\mc U)$ is a factor of $\mc P(\mc U) / \mc S_{\mbox{\tiny comm}}$, it must also be a perfect central extension of $G$.

\end{proof}

The topologies on the groups $\mc P(\mc U)$ and $\mc C^G(\mc U)$ are very unwieldy: while they have all of the nice separation properties, it lacks the key property of local compactness (since, even if a sequence of paths may stay very close to the identity, if the combinatorial length of the paths tends to $\infty$, they cannot converge). We have no guarantee, then, that the groups $\mc P(\mc U) / \mc S(\mc U)$ and $\mc C^G(\mc U) / \mc S(\mc U)$ are locally compact. The lack of local compactness makes many of the arguments that work in topological group theory fail.

We do not know whether the (normal closure of the) stable cycles $\mc S(\mc U)$ is a closed subgroup of $\mc C(\mc U)$ (in fact, we will see soon that it is not). Thus, the quotients $\mc C^G(\mc U) / \mc S(\mc U)$ and $\mc P(\mc U) / \mc S(\mc U)$ equipped with the quotient topology may not be Hausdorff. In fact, the group $\mc C^G(\mc U) / \mc S_{\mbox{\tiny comm}}$ in all computed cases for simply connected groups contains an uncountably infinite-dimensional vector space over $\Q$ as a factor (see eg, \cite{sah-wagoner}), but no point can be separated from the identity. That is, $\mc C^G(\mc U) / \mc S_{\mbox{\tiny comm}} = \overline{\set{e}}$, where the closure is taken in the group $\widetilde{G} = \mc P(\mc U) / \mc S_{\mbox{\tiny comm}}$.

This technique was used in previous arguments for local rigidity. Indeed, if we have that $\mc C^G(\mc U) / \mc S(\mc U)$ has the indiscrete topology, then any continuous homomorphism or action by this group is trivial. In place of this, we will show a weaker property of $\mc C^G(\mc U) / \mc S(\mc U)$ (Lemma \ref{key-lemma}) and that it suffices for the proofs of cocycle rigidity.

\begin{lemma}
\label{section-exist}
The projection $\mc P(\mc U) / \mc S(\mc U) \to G$ has a continuous local section
\end{lemma}

\begin{proof}
If we can find a continuous local section for $\mc P(\mc U) \to G$, then by composing with the projection $\mc P(\mc U) \to \mc P(\mc U)/\mc S(\mc U)$, we get the desired section. In the same vein, because the topology on $\mc P(\mc U)$ is the quotient of $\bigsqcup_{n \in \N} \mc P_n(\mc U)$, if we can find a section for one of the $\mc P_n(\mc U) \to G$, then we are done. But $\mc P_n(\mc U) = \left( \bigcup_{r \in \Delta} U_r \right)^n = \bigcup_{(r_1,\dots,r_n) \in \Delta^n} U_{r_1} \times \dots \times U_{r_n}$. For sufficiently large $n$, the projection $\mc P_n(\mc U) \to G$ is surjective. But the multiplication from $U_{r_1} \times \dots \times U_{r_n}$ is not only continuous but smooth. Furthermore, it is homogeneous, so surjectivity implies that it is a submersion. But submersions have continuous local sections, so we are done.
\end{proof}

We state here for reference later a fundamental result for central extensions of Lie groups:

\begin{proposition}
\label{central-splits}
Let $G$ be a connected, simply connected, semisimple Lie group, and $A$ an abelian Lie group. Let $H$ be a topological central extension of $G$ by $A$ with a continuous local section:

\[ 1 \to A \to H \to G \to 1\]

Then there is a splitting homomorphism $\sigma : G \to H$ (Thus, $H \cong G \times A$, and the projection $H \cong G \times A \to G$ is onto the first coordinate)
\end{proposition}

\begin{proof}
We first conclude that $H$ is a Lie group, by the classical work of Gleason-Montgomery-Zippin (since there is a local section, $H$ is locally Euclidean). First, quotient by the connected component of $A$ in $H$ to arrive at an extension of $G$ by a discrete group $A / A^\circ$. Since $G$ is simply connected, there is a splitting homomorphism ($H / A^\circ$ must have the same Lie algebra that $G$ does). Hence $H \cong H_0 \times A/A^\circ$, and hence we may assume that $A$ and $H$ are connected.

Finally, once we have that $A$ is connected, we note that the extension splits only if the induced sequence of Lie algebras splits. The splitting of this is guaranteed by the classical Levi splitting theorem.

\end{proof}

The scheme for proving the periodic cycle functional vansihes is to pit the splitting theorem of Levi (which takes the form of Proposition \ref{central-splits}) against the perfectness of the universal central extension, leading to a contradiction.

\begin{definition}
A topological abelian group is {\it minimally almost periodic} if it has no continuous homomorphisms into locally compact groups
\end{definition}

\begin{lemma}[Key Lemma]
\label{key-lemma}
If $G$ is simply connected, $\mc C^G(\mc U) / \mc S(\mc U)$ is minimally almost periodic
\end{lemma}

\begin{proof}
Let $f$ be a continuous homomorphism from $\mc C^G(\mc U)$ to a locally compact group $H$. By pullback, consider it as a homomorphism from $K = \mc C^G(\mc U) / \mc S_{\mbox{\tiny comm}}$. Let $f'$ denote this pullback, which will be trivial if and only if $f$ is. See the diagram below:

\[
\begin{tikzcd}[column sep=normal]
1 \arrow{r}{} & K \arrow{r}{} \arrow{d}{} \arrow{drr}[near end]{f'}& \widetilde{G} \arrow{r}{} & G \arrow{r}{} & 1 \\
&  \mc C^G(\mc U) / \mc S(\mc U) \arrow{r}{f} & Z \arrow[dashed]{r}{} & \R / \Z
\end{tikzcd}
\]

Since $K$ is central in $G$, $Z = \overline{f( \mc C^G(\mc U) / \mc S(\mc U))}$ must be abelian. By composing with a character of $Z$ (ie, a homomorphism $\varphi : Z \to \R / \Z$ which is nontrivial on $f'(K)$), we can without loss of generality assume that $Z$ is a subgroup of the circle (since $Z$ is locally compact, characters separate points). Consider the group $G_1 = \widetilde{G} / \ker f'$, and $K_1 = K / \ker f'$. Then $f'$ induces a continuous isomorphism from $K / K_1$ to $f'(K)$. Since the group $Z$ is Hausdorff, $\overline{\set{e}} \subset \ker f'$ and the group $G_1$ must Hausdorff as well. Now, consider $G_1 \times Z$, and the subgroup $K_1^\Delta = \set{ (k, f(k)) : k \in K_1} \subset K_1 \times Z \subset G_1 \times Z$ (the ``diagonal'' embedding of $K_1$ into $K_1 \times Z$). $K_1^\Delta$ is closed, because $K_1$ is closed in $G_1$. It is central (and hence normal), since $K_1$ is central in $G_1$. Note that $(K_1 \times Z) / K_1^\Delta$ is isomorphic to $Z$ via the surjective homomorphism $(x,y) \mapsto f'(x)^{-1}y$. So $G_2 = (G_1 \times Z) / K_1^\Delta$ is a topological central extension of $G$ by $Z$ ($G_2$ can be thought of the completion of $G_1$ with respect to the embedding of $K$ into $H$). Recall that there is a local continuous section for $\widetilde{G} \to G$ (and hence for $G_2 \to G$) by Lemma \ref{section-exist}.

\[
\begin{tikzcd}
1  \ar{r}{} & K  \ar{r}{} \ar{d}[swap]{/ \ker f'} & \widetilde{G} \ar{r}{} \ar{d}[swap]{/ \ker f'} & G \ar{r}{} \arrow[equals]{d}{} & 1 \\
1 \ar{r}{} & K_1 \ar{r}{} \ar{d}[swap]{\times Z} & G_1 \ar{r}{} \ar{d}[swap]{\times Z} & G \ar{r}{} \ar[equals]{d} & 1 \\
1 \ar{r}{} & K_1 \times Z \ar{r}{} \ar{d}[swap]{/ K_1^\Delta} & G_1 \times Z \ar{r}{} \ar{d}[swap]{/K_1^\Delta} & G \ar{r}{} \ar[equals]{d}{} & 1 \\
1 \ar{r}{} & Z \ar{r}{} & G_2 \ar{r}{} & G \ar{r}{} & 1
\end{tikzcd}
\]

 Since $G$ is simple and $Z$ is a closed subgroup of $\R / \Z$ (and hence a Lie group), the sequence $1 \to Z \to G_2 \to G \to 1$ splits as $G_2 \cong G \times Z$ by Proposition \ref{central-splits}.

Let $\varphi$ be the obvious map which embeds $G_1$ in $G_2$. Since $G_1$ is perfect, $\varphi(K_1) \subset \varphi\left([G_1,G_1]\right) \subset \left[\varphi(G_1),\varphi(G_1)\right] \subset [G_2,G_2] = G$. But by construction $\varphi(K_1) \subset Z$ (since it must be taken into the kernel of the projection of $G_2 \cong G \times Z \to G$), so $\varphi(K_1) \subset G \cap Z = \set{e}$. Since $\varphi$ is an embedding, we conclude that $K_1$ is trivial. So $f'$ and hence $f$ is trivial.

\end{proof}

\section{Holonomy Projections of Cycles}
\label{holonomy-sec}

In section \ref{lyap-sec}, the existence of two families foliations are asserted: $\set{\mc U_r}$ and $\set{\mc F_r}$, representing the coarse Lyapunov foliations for $\alpha_0$ and $\alpha$, respectively. Note that the foliations $\mc F_r$ are only H\"older: we first construct the coarse Lyapunov foliations $\overline{\mc F}_r$ for $\overline{\alpha}$, and taken to H\"older foliations $\mc F_r = h( \overline{\mc F}_r )$, where $h$ is the map provided by Theorem \ref{normal-stability}.

Let us work on the universal cover. Let $\mc P_x(\mc U)$ represent the space of paths using the family of foliations $\mc U$ based at $x$ (and similarly for $\mc P_x(\mc F)$). We will now establish the existence of holonomy projections $P^{x,y} : \mc P_x(\mc U) \to \mc P_y(\mc F)$, provided $y \in Nx$. For a detailed discussion, see \cite{kd2011}. For a path with a single leg $\rho = (x,z)$ with $z \in U_r(x)$, define:

\[ P^{x,y}(\sigma) = (y, \mc F_r(y) \cap Nz) \]

The definition makes sense locally, and can be extended to any $z \in U_r(x)$ by using the normally hyperbolic dynamics to bring it back to a neighborhood of $x$. Let us extend this definition inductively for a Lyapunov path $\sigma = (x_0,\dots,x_n)$. If $P^{x,y}(x_0,\dots,x_{n-1}) = \tau = (y_0,\dots,y_{n-1})$, let:

\[ P^{x,y}(\sigma) = (y_0,\dots,y_{n-1},\mc F_r(y_{n-1}) \cap Nx_{n-1}) \]

where $x_n \in \mc F_r(x_{n-1})$. Let us fix the basepoint $e \in G$ and denote $P = P^{e,e}$. Before proving Proposition \ref{u-cycles-proj}, we need a Definition and Lemma:

\begin{proposition}
\label{u-cycles-proj}
$P(\mc C^c(\mc U)) \subset \mc C^c(\mc F)$
\end{proposition}

\begin{proof}
We will first show it for stable cycles. If $\sigma$ is a stable cycle at $e$, then its projection $P(\sigma)$ will be a Lyapunov path, and by construction of the projections, the endpoint of $P(\sigma)$ will be on the leaf $N$. But if we then begin to apply an element of our action which contracts the entire Lyapunov path (it is stable), it will contract all of the legs with exponential speed faster than the difference of the endpoints in $N$ (since the action is normally hyperbolic with respect to $N$). This is a contradiction.

So $P(\mc S(\mc U)) \subset \mc S(\mc F)$. Given a path $\rho \in \mc P(\mc U)$, it can be considered as a path from any basepoint by right translation (see the group structure of $\mc P(\mc U)$ in Section \ref{schur-sec}). This defines an action of $\mc P(\mc U)$ on $M$ by letting the path $\rho$ act by sending a point to the endpoint of its canonical projection $P(\rho)$. Since $\mc S(\mc U)$ acts trivially, this can be considered as an action of $\mc P(\mc U) / \mc S(\mc U) = \widetilde{G}$. We have shown that this action preserves the foliation into cosets $Nx$, and can be considered as an action on the universal cover, which we will now do.

Now consider the subgroup of $\mc P(\mc U)$ generated by all paths which end in $N$, which includes the group $\mc C^c(\mc U)$ (call this group $\widetilde{N}$). The group action of $\widetilde{N}$ by $\mc F$-holonomies is a $C^0$-perturbation of the cycle group acting trivially and $N$ acting by translations, so it must still act transitively on each leaf. Fix a leaf $Nx$, and consider the stabilizer of the $\mc F$-holonomy action $S = \Stab_{\mc F}(x)$.

For convenience, let $Z$ denote the group $\mc C^c(\mc U) / \mc S(\mc U)$ (ie, the group of $\mc U$-cycles modulo stable $\mc U$-cycles, which is a factor of the Schur multiplier of $G$). Furthermore, let $H = \overline{Z \cdot A}$. We claim that $\widetilde{N} / S$ is homeomorphic to $N$, since $\widetilde{N}$ acts transitively by $\mc F$-holonomies. This follows from the stability of local transitivity (Proposition \ref{local-transitive}), so an open set of $\mc U$-paths ending on $N$ send any point to a neighborhood of that point in $N$ via $\mc F$-holonomies (this implies that the evaluation maps is open, so modulo the stabilizer, it is open, continuous and bijective). Since $\widetilde{N} / S$ is also an algebraic quotient, we may let $H$ act by translation on the left, so that the orbits become the quotient $H/S$. 

We claim that $H / S$ is closed in $\widetilde{N} / S$. Suppose that $h_nS \to gS$ in $\widetilde{N}/S$. Then there exist points $s_n$ such that $h_ns_n \to g$ in $\widetilde{N}$. But every point of $H$ can be written as $h_n = \lim_{m \to \infty} a_m^{(n)}s_m^{(n)}$, since $Z$ is central and by definition of $H$. But then \[g = \lim_{\substack{m \to \infty \\ n \to \infty}} a_m^{(n)} (s_m^{(n)}s_n)\] and so $g \in H$. Furthermore, while $\widetilde{N}/ S$ may not be a group, we claim that, $H / S$ is an group. We must show that $S$ is normal in $H$. This follows from the fact that $Z$ is central, and hence if $h = \lim_{n \to \infty} a_ns_n \in H$ and $s \in S$, $hsh^{-1} = \lim_{n \to \infty} a_ns_nss_n^{-1}a_n^{-1} = \lim_{n \to \infty} s_nss_n^{-1} \in S$, since $S$ is closed. So $H / S$ has a group structure. 



But $H / S$ is also a closed subset of $\widetilde{N} / S$, which is homeomomrphic to $N$, and hence locally compact. If $Z \not\subset S$, then the obvious homomorphism $Z \hookrightarrow H \to H / S$ is nontrivial, and we contradict Lemma \ref{key-lemma}.

\end{proof}

\begin{corollary}
$S / Z$ is a discrete group
\end{corollary}

\begin{proof}
We continue with the notations of the proof of Proposition \ref{u-cycles-proj}. Note that since $Z$ is central, and contained in $S$, $S/ Z \subset \widetilde{N} / Z = N$ is a subgroup. But $N / (S/Z)$ must be homeomorphic to $\widetilde{N} / S$, which is homeomorphic to $N$. But the only way for this to happen is to not lose dimension, so $S / Z$ must be discrete.
\end{proof}

Analogous projections $Q^{x,y} : \mc P_x(\mc F) \to \mc P_y(\mc U)$ can be defined in an obvious way. It is also clear that $P^{x,y} \of Q^{y,x} = \id$ and $Q^{x,y} \of P^{y,x} = \id$. We use the similar convention of setting $Q = Q^{e,e}$.

\begin{lemma}
\label{small-f-project}
There exists $\delta> 0$ such that if $\sigma \in \mc C^c(\mc F)$ is contained in a ball of radius $\delta$ in $G$, then $Q(\sigma) \in \mc C^c(\mc U)$
\end{lemma}

\begin{proof}
Recall that a the Lyapunov foliations $\mc F$ are locally transitive, hence given $\ve > 0$, there exists a $\delta > 0$ such if $x \in B_\delta(e)$, then there exists a path $\rho_x$ with a bounded length and number of legs which is contained in $B_\ve(e)$ which ends at $x$. But then if $\sigma = (x_0,x_1,\dots,x_n)$, and $\tau_i = (x_i,x_{i+1})$ are the paths representing the legs, we know:

\[ \sigma = (\tau_1 * \bar \rho_{x_1}) * (\rho_{x_1} * \tau_2 * \bar \rho_{x_2}) * (\rho_{x_2} * \tau_3 * \bar \rho_{x_3}) * \dots * ( \rho_{x_n} * \tau_n ) \]

But each conjugate $\rho_{x_{i-1}} * \tau_i * \bar\rho_{x_i}$ is a path of bounded combinatorics which stays inside $B_\ve(e)$. Hence the projection as a $\mc U$-path must end close to the identity if $\ve$ is sufficiently small. But since $S / Z$ is discrete, if a $\mc U$-path which projects to an $\mc F$-cycle (that is, it is in $S$) ends close to the identity, it must in fact be a $\mc U$-cycle (that is, it must be in $Z$). Hence, each of the paths $\rho_{x_{i-1}} * \tau_i * \bar \rho_{x_i}$ project to $\mc U$-cycles and so does $\sigma$.
\end{proof}

\begin{proposition}
\label{f-cycles-proj}
$Q(\mc C^c(\mc F)) \subset \mc C^c(\mc U)$
\end{proposition}

\begin{proof}
We continue with the notations in the proof of Proposition \ref{u-cycles-proj}. Since $Q$ is the inverse of $P$, this is equivalent to showing that every $\mc F$-cycle is realized by a $\mc U$-cycle. Note that every $\mc F$-cycle is realized by a $\mc U$-path ending on $N$, since we can use the canonical projection $Q$. Saying that every $\mc F$-cycle is realized as a $\mc U$-cycle is to say that $Z = S$ (where $Z = \ker (\widetilde{G} \to G)$ is the group of $\mc U$-cycles and $S = \Stab_{\mc F}(x)$ under the action of $\widetilde{N}$ by $\mc F$-holonomies).

Note that elements of $S / Z \subset N$ represent the possible endpoints of the projections of $\mc F$-cycles. We wish to show that $S / Z = \set{e}$, so that $S = Z$.

If $\sigma \in \mc C(\mc F)$ is a contractible cycle, then it can be traced out by a $C^0$-path $\gamma_0 : [0,1] \to G$. Furthermore, there exists a homotopy $\gamma_t$ which contracts $\gamma_0$. Let $\delta$ be as in Lemma \ref{small-f-project}. Furthermore, any $C^0$-path can be $\delta/2$ approximated by a $\mc F$-path on the manifold (even though the combinatorics can become very long).

But if we divide our homotopy into small boxes of the form $[k/n,(k+1)/n] \times [l/n,(l+1)/n]$, we may choose $n$ large enough so that each box is contained inside a ball of radius $\delta/2$. Furthermore, each such edge can be realized by an $\mc F$-path, losing $\delta/2$ of closeness, but thus guaranteeing each corresponding $\mc F$-cycle lies in a $\delta$ ball. But then each box represents an $\mc F$-cycle contained in a ball of radius $\delta$, so projects to a $\mc U$-cycle. So the original cycle $\sigma$ projects to a $\mc U$-cycle.



\end{proof}

Propositions \ref{u-cycles-proj} and \ref{f-cycles-proj} give us the following, which guarantees that if we show cocycle rigidity for $\alpha_0$, then we also get it for $\alpha$.

\begin{proposition}
\label{holonomy-iso}
$P$ induces a continuous isomorphism between $\mc C^c(\mc U)$ and $\mc C^c(\mc F)$. Furthermore, \[P(\mc S(\mc U)) = \mc S(\mc F)\]
\end{proposition}

\begin{corollary}
\label{simple-transitive}
If $G$ is simply connected, the group of $\mc F$ holonomies is isomorphic to a Lie group
\end{corollary}

\begin{proof}
It acts transitively, because it is a perturbation of the foliation family $\mc U$, which obviously act (simply) transitively.

If we choose 2 different $\mc F$-paths $\rho_1$ and $\rho_2$ to carry a point $x$ to a point $y$, then the path $\rho_1 * \overline{\rho_2}$ would be a cycle at $x$. By Proposition \ref{f-cycles-proj}, this projects to a $\mc U$-cycle. But as shown in the proof of Proposition \ref{u-cycles-proj}, the projection of any $\mc U$-cycle acts trivially on $Nx$, and hence projecting back to now a family of $\mc F$-cycles, we see that taking $\rho_1$ or $\rho_2$ leads to the same map on $Nx$. Thus, the group of $\mc F$-holonomies is isomorphic to a Lie group (since it acts simply transitively on a manifold, it is a manifold itself).
\end{proof}

\begin{proof}[Proof of Proposition \ref{homeo}]
We pass to the universal cover. Then the map $h(x) = H(x)^{-1}h'(x)$ still preserves the leaves $Nx$, and being $C^0$-close to the identity, is surjective on each leaf. Hence we need only show injectivity. Suppose that $h(x) = h(y)$. Since on the universal cover the $\mc F$-holonomies act simply transitively on $Nx$, there exists a unique $\mc F$-holonomy $F : Nx \to Nx$ such that $F(x) = y$. Since $h$ is a semiconjugacy, and takes Lyapunov foliations to Lyapunov foliations, we immediately get that $h \of F = h$. But since $h$ is close to the identity, the preimage $h^{-1}(h(x))$ must be contained in a small neighborhood of $x$.

But the holonomy group is a Lie group, because it acts simply transitively on a manifold (and is hence a manifold itself) and can have no small subgroups. We have arrived at a contradiction, provided our perturbation is sufficiently small.
\end{proof}

\section{A proof of Theorem \ref{main}}
\label{proofs-sec}

Throughout this section, we will assume the group $G$ is simply connected, without loss of generality (we may always lift the lattice to the universal cover).

\begin{theorem}
\label{contractible-vanishing}
If $G$ is a simply connected, simple Lie group, and $\beta$ is a sufficiently small H\"older cocycle over $\alpha_0$, the periodic cycle functional $P_\beta$, restricted to $\mc C_e^G(\mc U)$, is trivial.
\end{theorem}

\begin{proof}
This is a direct application of Lemmas \ref{stable-vanishing} and \ref{key-lemma}.
\end{proof}

\begin{proposition}
\label{simple-case}
If $G$ is a simple Lie group, then the conclusions of Theorems \ref{cocyle-rig} and \ref{main} hold
\end{proposition}

\begin{proof}
The discussion preceding the statement of Theorem \ref{cocyle-rig} (ie, Propositions \ref{correction-cocycle} and \ref{homeo}) tells us that proving Theorem \ref{cocyle-rig} is sufficient for Theorem \ref{main}. We use the criterion presented in Proposition \ref{pcf-suff} to show rigidity of cocycles. Proposition \ref{holonomy-iso} allows us to consider the foliations associated to the action $\alpha_0$, rather than $\alpha$. Theorem \ref{contractible-vanishing} guarantees that it vanishes on the contractible cycles, so the periodic cycle functional will induce a function on $\mc C_e(\mc U) / \mc C_e^c(\mc U) = \pi_1(M) = \Gamma$. That is, it induces a homomorphism from $\Gamma$ to the centralizer of the Cartan. But by Corollary \ref{lattice-finite}, such a homomorphism has its image discrete and finite with fixed bounded order. Thus, by reducing the size of the cocycle (which can be obtained by reducing the size of the perturbation), we can guarantee its image is trivial.

\end{proof}

\begin{remark}
While this argument suffices to show that the periodic cycle functional vanishes, more is true. In \cite{dupont-parry-sah}, it is shown that the induced map on $H_2(G) \to H_2(SL(d,\C)) = K_2(\C)$ is injective when $G$ is absolutely simple and not of type E or F. In the original work for partially hyperbolic systems, it was shown that $\mc C(\mc U) / \mc S(\mc U) \cong \overline{\set{e}}$ for $SL(d,\C)$. Combining these results gives that $H_2(G) = \overline{\set{e}}$ in the topology of the universal central extension by continuity (ie, $\mc C^c_e(\mc U) = \overline{\mc S_e(\mc U)}$). Showing that the stable cycles were dense in the space of contractible cycles is the method used in \cite{kd2011}, \cite{damjanovic07}, \cite{zwang-1} and \cite{zwang-2}. Previous authors faced the difficulty of showing this directly, employing an argument similar to that of Milnor's original arguments as it appears in \cite[Theorem A.1]{milnor-intro} to show $\mc S_e(\mc U)$ is dense in $\mc C_e(\mc U)$. The approach presented here circumvents this problem by showing the vanishing of the periodic cycle functional without considering ``how large'' the normalizer of the stable cycle group is directly.

As a consequence of the previous method, we get a statement stronger than usual about extensions of Lie groups (to the author's best knowledge, the condition of local compactness of the extension was always needed previously):

\begin{proposition}
If $G$ is an absolutely simple, simply connected Lie not of type E or F, then there does not exist a nontrivial perfect Hausdorff central extension of $G$. 
\end{proposition}
\end{remark}

\begin{proof}[Proof of Theorem \ref{main} and Theorem \ref{cocyle-rig}]
Let $G$ be semisimple, with all its simple factors of real rank greater than 2, so that $G = \prod_i G_i$. Then any split Cartan subalgebra $\mf h \subset \mf g = \bigoplus_i \mf g_i$ decomposes as $\mf h = \bigoplus_i \mf h_i$, where $\mf h_i$ is a split Cartan subalgebra for $\mf g_i$. If $\Delta_i$ are the set of roots for $\mf g_i$, then we get a splitting:

\[ \mf g = \mf h \oplus \bigoplus_i \bigoplus_{r \in \Delta_i} \mf g_r \]

Furthermore, if $r_1 \in \Delta_i$ and $r_2 \in \Delta_j$, and $r_1 + r_2 \in \Delta$, then $i = j$, and $r_1 + r_2 \in \Delta_i$. This implies, by the genericity assumption \gen, that if $g_1 \in U_{r_1}$ and $g_2 \in U_{r_2}$, then $[g_1,g_2] \in \mc S(\mc U)$. So in the group $\mc C_e^c(\mc U) / \mc S(\mc U)$, any element can be written uniquely as a product of cycles in each $G_i$. That is, $\mc C_e^c(\mc U)/ \mc S_e(\mc U) = \prod_i \mc C_e^{c,i}(\mc U) / \mc S_e^i(\mc U)$, where $\mc C_e^{c,i}(\mc U)$ is the contractible cycle group for the group $G_i$.

The periodic cycle functional, restricted to each component $\mc C_e^{c,i}(\mc U)$, must be trivial, hence it is trivial on all of $\mc C_e^c(\mc U)$ by Proposition \ref{simple-case} (note that we may again substitute the foliation family $\mc F$ for $\mc U$ freely by Proposition \ref{holonomy-iso}). By Proposition \ref{lattice-rig}, we get that for sufficiently small cocycles, the periodic cycle functional vanishes. By Proposition \ref{pcf-suff} and Proposition \ref{holonomy-iso}, we get the result.

\end{proof}

\section{Comments on the Remaining Cases}
\label{next-steps-sec}

The results of this paper essentially settle the rigidity program for restrictions of Cartan actions. However, a few persistent cases exist. One may try to consider actions of abelian subgroups that are not ``diagonalizable.'' That is, an abelian subgroup which does not fit inside a Cartan subalgebra for a (semi)simple group, but still has its Lyapunov spaces generating the Lie algebra.

Suspensions of toral automorphism actions have been shown to be rigid. Indeed, it is more conducive to study the first-return action of $\Z^k$ (ie, the toral automorphism actions themselves). The only condition required here is ``no rank one factors'' (see \cite{kdtoral}), which is weaker than condition \gen. Furthermore, in the presence of rank one factors, the action fails to be $C^{\infty,k,\infty}$-rigid. It is hoped that a similar characterization may exist for actions on semisimple groups, as it is known that the genericity assumption \gen \, is not a necessary condition for rigidity \cite{zwang-2}.

Automorphism actions on nilpotent groups remain open. The main difficulty in this case is that the standard assumptions of no rank one factors does not imply that there are no rank one subgroups, due to the nonabelian nature of the phase space.

The semisimple and toral cases represent the extreme ends of the spectrum as described by the Levi decomposition: $G \cong S \ltimes R$, where $S$ is semisimple and $R$ is solvable (the solvable radical). The semidirect product is given by a representation $\rho : S \to \Aut(\Lie(R)) \subset GL(\Lie(R))$. Any lattice $\Gamma$ in $G$ will respect this semidirect product. That is, $\Gamma \cong \Gamma_0 \ltimes \Lambda$, where $\Gamma_0$ is a lattice in $S$ and $\Lambda$ is a lattice in $R$.

Let $R$ be abelian, and $\rho(\Gamma_0)$ contain an Anosov element (which is equivalent to saying there are no simple summands on which $\rho(\Gamma_0)$ acts by a finite group of isometries). Furthermore, assume that if $\chi$ is a weight of $\rho$ and $r$ is a root of the group $G$ with respect to the same split Cartan subalgebra, then $\chi \not= \lambda r$ for every $\lambda \in \R$ (such representations are {\it non-resonant}). A (partially hyperbolic) twisted Weyl chamber flow is the action of a Cartan action (ie, a generic restriction of the action of a Cartan subalgebra in $S$) on $M = G / \Gamma$. Zhenqi Wang has shown that generic restrictions of such actions (in the sense of \gen) are $C^{\infty,k,\infty}$-rigid for sufficiently large $k$ in the case when $S$ is (semi)split and rigid in the universal cover in the case when $S$ is any semisimple Lie group \cite{zwang-twisted}. Her method relies on choosing nice $G$-orbits inside $M$ which are exactly Weyl chamber flows and correcting the cocycle in each such orbit using methods presented here. The resulting cocycle is trivialized by Zimmer cocycle rigidity. 
The problem then has only a handful of actions resistant to current methods:

\begin{itemize}
\item Actions which satisfy assumptions weaker than \gen
\item Higher-rank automorphism actions on Nilpotent groups
\item ``Non-diagonalizable'' actions in the semisimple and toral cases
\item The twisted case when:
\begin{itemize}
\item  the representation $\rho|_{\Gamma_0}$ has no Anosov element\footnote{It may be the case that no such ergodic actions exist when $G$ is genuinely higher-rank}
\item the representation $\rho$ is resonant
\item the group $R$ is nonabelian
\item $G$ is not semi-split\footnote{Such actions have been shown to be rigid in the universal cover \cite{zwang-twisted}}
\end{itemize}
\end{itemize}

\bibliographystyle{plain}
\bibliography{central}

\end{document}